\documentclass[11pt]{article}

\usepackage[english]{babel}
\usepackage{upref,amsfonts,amsxtra,latexsym,color}
\usepackage{amssymb,amsmath,amsthm,a4wide,epsf,mathrsfs,verbatim,hyperref}
\usepackage[all]{xy}
 
\newtheorem{theorem}{Theorem}[subsection]
\newtheorem{lemma}[theorem]{Lemma}
\newtheorem{corollary}[theorem]{Corollary}
\newtheorem{proposition}[theorem]{Proposition}

\theoremstyle{definition}

\newtheorem{remark}[theorem]{Remark}
\newtheorem{example}[theorem]{Example}

\numberwithin{equation}{section}
\numberwithin{theorem}{section}

\newcommand{\Tr}{\mathrm{Tr}}

\newcommand{\cM}{{\cal M}}

\newcommand{\cA}{{\cal A}}

\newcommand{\cB}{{\cal B}}

\newcommand{\cT}{{\cal T}}

\newcommand{\cQ}{{\mathcal Q}}
\newcommand{\cZ}{{\mathcal Z}}
\newcommand{\cR}{{\mathcal R}}
\newcommand{\cFM}{{\mathcal{FM}}}
\newcommand{\C}{{\mathbb C}}
\newcommand{\Z}{{\mathbb Z}}
\newcommand{\N}{{\mathbb N}}

\newcommand{\R}{{\mathbb R}}
\newcommand{\Cs}{{$C^*$-al\-ge\-bra}}

\newcommand{\sh}{{$^*$-ho\-mo\-mor\-phism}}

\date{}
\title{Factorizable maps and traces on the universal free product of matrix algebras}

\author{Magdalena Musat$^*$ and Mikael R\o rdam\thanks{This research was supported by a travel grant from the Carlsberg Foundation, and by a research grant from the Danish Council for Independent Research, Natural Sciences. This work was in part carried while the authors were visiting the Institute for Pure and Applied Mathematics (IPAM), which is supported by the National Science Foundation.}}

\begin{document}

\maketitle

\begin{abstract} 
\noindent We relate factorizable quantum channels on $M_n(\C)$, for $n \ge 2$, via their Choi matrix, to certain matrices of correlations, which, in turn, are shown to be parametrized by traces on the unital free product $M_n(\C) *_\C M_n(\C)$. Factorizable maps that admit a finite dimensional ancilla are parametrized by finite dimensional traces on $M_n(\C) *_\C M_n(\C)$, and factorizable maps that approximately factor through finite dimensional \Cs s are parametrized by traces in the closure of the finite dimensional ones. The latter set of traces is shown to be equal to the set of hyperlinear traces on $M_n(\C) *_\C M_n(\C)$.  We finally show that each metrizable Choquet simplex is a face of the simplex of tracial states on $M_n(\C) *_\C M_n(\C)$.
\end{abstract}

\section{Introduction}

\noindent Factorizable maps were introduced by C.\ Anantharaman-Delaroche in \cite{A-D:factorizable} in her study of non-commutative analogues of classical ergodic theory results. This notion has lately found interesting applications in quantum information theory, e.g., in solving in the negative the asymptotic quantum Birkoff conjecture, \cite{HaaMusat:CMP-2011}. A factorizable channel $T$ on  $M_n(\C)$ is a unital completely positive trace-preserving map $M_n(\C) \to M_n(\C)$ that \emph{factors} through a finite tracial von Neumann algebra $(M,\tau_M)$ via two unital \sh s $M_n(\C) \to M$ (see more details in Section~\ref{sec:Choi}). Factorizable maps were in \cite{HaaMusat:CMP-2011} equivalently characterized as arising from
 an ancillary tracial von Neumann algebra $(N,\tau_N)$ and a unitary $u$ in $M_n(\C) \otimes N$ such that $T(x) = (\mathrm{id}_n \otimes \tau_N)(u(x \otimes 1_N)u^*)$, for all $x \in M_n(\C)$. It was recently shown in \cite{MusRor:infdim} that the ancilla $N$ cannot always be taken to be finite dimensional (or even of type I). U.\ Haagerup and the first named author proved in  \cite[Theorem 3.7]{HaaMusat:CMP-2015} that  each factorizable channel can be approximated by factorizable ones possessing a finite dimensional ancilla if and only if the Connes Embedding Problem has an affirmative answer.

In this paper we present a different viewpoint on factorizable channels, that bears resemblance to the description of quantum correlation matrices arising in Tsirelson's conjecture. In Section~\ref{sec:Choi}, we establish when a linear map on $M_n(\C)$ is factorizable in terms of certain properties of its Choi matrix. Fritz, \cite{Fritz:Tsirelson}, and, independently, Junge et al., \cite{JNPP-GSW:Tsirelsen}, expressed the quantum correlation matrices appearing in Tsirelson's conjecture in terms of states on the minimal, respectively, the maximal tensor product of the full group \Cs{} associated with the free product of finitely many copies of a finite cyclic group. (This characterization, in turn, was the bridge needed to prove the equivalence between Tsirelson's conjecture and the Connes Embedding Problem, with the finishing touch provided by Ozawa, \cite{Ozawa:Tsirelson}.)  In a similar spirit,  we recast the description of factorizable maps on $M_n(\C)$ in terms of traces on the unital universal free product $M_n(\C) *_\C M_n(\C)$. We show that factorizable channels with finite dimensional ancilla are parametrized by finite dimensional traces on $M_n(\C) *_\C M_n(\C)$, and factorizable channels that can be approximated by ones possessing a finite dimensional ancilla are parametrized by traces in the closure of the finite dimensional ones. 

This new viewpoint on factorizable channels led us to further analyze the trace simplex of the unital universal free product $M_n(\C) *_\C M_n(\C)$. This \Cs{}  is known to be residually finite dimensional, \cite{ExelLoring:free_products},  and semiprojective, \cite{Bla:shape}. As explained in  Section~\ref{sec:fd}, it is not the case that the set of finite dimensional traces on a residually finite dimensional \Cs{} necessarily is weak$^*$-dense. In this section, we further review known facts and establish new results on the closure of finite dimensional traces on general (residually finite dimensional) unital \Cs s, and we relate this set to the convex compact sets of quasi-diagonal, amenable, respectively, hyperlinear traces. For the particular \Cs{} 
$M_n(\C) *_\C M_n(\C)$, we show in Theorem~\ref{thm:Tfin(MnMn)}  that the closure of the finite dimensional traces is equal to the set of hyperlinear traces. By the aforementioned result, \cite[Theorem 3.7]{HaaMusat:CMP-2015}, by Haagerup and the first named author, if the finite dimensional tracial states on $M_n(\C) *_\C M_n(\C)$ are weak$^*$-dense in the set of all tracial states, then the Connes Embedding Problem has an affirmative answer. Theorem~\ref{thm:Tfin(MnMn)} implies that the converse also holds, so these two statements are equivalent.

We further show that whenever $\cA$ is a unital \Cs{} generated by $n-1$ unitaries, then $M_n(\C) \otimes \cA$ is a quotient of $M_n(\C) *_\C M_n(\C)$, and therefore generated by two copies of $M_n(\C)$.  For $n \ge 3$, the  class of \Cs s $\cA$ as above includes all singly generated \Cs s (and hence, for example, all finite dimensional \Cs s). As an interesting application, we show in Theorem~\ref{prop:Poulsen} that the Poulsen simplex is a face of the trace simplex of $M_n(\C) *_\C M_n(\C)$, whenever $n \ge 3$. We leave open if the trace simplex of $M_n(\C) *_\C M_n(\C)$ itself is the Poulsen simplex. We recommend Alfsen's book, \cite{Alf:convex}, as an excellent reference for Choquet theory.

\section{Finite dimensional traces and their convex structure} \label{sec:fd}

\noindent
Let $\cA$ be a unital \Cs, and denote by $T(\cA)$ the simplex of tracial states on $\cA$. For each $\tau \in T(\cA)$ consider the closed two-sided ideal
\begin{equation} \label{eq:I-tau}
I_\tau = \{a \in \cA : \tau(a^*a)=0\}
\end{equation}
in $\cA$. A tracial state $\tau$ on $\cA$ is said to \emph{factor through} another unital \Cs{} $\cB$, if $\tau = \tau' \circ \varphi$, for some unital \sh{} $\varphi \colon \cA \to \cB$ and some tracial state $\tau'$ on $\cB$. If $\varphi$ is surjective, we say that $\tau$ \emph{factors surjectively through} $\cB$.  Furthermore,  $\tau$ is said to be  \emph{finite dimensional} if it factors through a finite dimensional \Cs. Equivalently, $\tau$ is finite dimensional  if and only if $\cA/I_\tau$ is finite dimensional. This, again, is equivalent to the enveloping von Neumann algebra $\pi_\tau(\cA)''$ of the GNS representation $\pi_\tau$ arising from $\tau$ being finite dimensional. (Note that $\pi_\tau(\cA) \cong \cA/I_\tau$.) The set of finite dimensional tracial states on $\cA$ is denoted by
$T_\mathrm{fin}(\cA)$. 
Clearly, $T_\mathrm{fin}(\cA)$  is non-empty precisely when $\cA$ admits at least one finite dimensional representation.  

\newpage

\begin{proposition} \label{prop:fdtrace1} Let $\cA$ be a unital \Cs, and assume that  $T_\mathrm{fin}(\cA)$ is non-empty. Then:
\begin{enumerate}
\item $T_\mathrm{fin}(\cA)$ is a  (convex) face of $T(\cA)$, and its closure $\overline{T_\mathrm{fin}(\cA)}$ is a closed face of $T(\cA)$.
\item $T_\mathrm{fin}(\cA) = \mathrm{conv} \Big(\partial_e T(\cA) \cap T_\mathrm{fin}(\cA)\Big)$, and $\partial_e T(\cA) \cap T_\mathrm{fin}(\cA)$ consists of those tracial states on $\cA$ that factor surjectively through $M_k(\C)$, for some $k \ge 1$. (Here $\partial_e T(\cA)$ denotes the set of extreme points of $T(\cA)$.)
\end{enumerate}
\end{proposition}

\begin{proof}  (i). Let $\tau_1, \tau_2$ belong to $T_\mathrm{fin}(\cA)$, witnessed by finite dimensional \Cs s $\cB_1$ and $\cB_2$, unital \sh s $\varphi_j \colon \cA \to \cB_j$, and tracial states $\sigma_j$ on  $\cB_j$ such that $\tau_j = \sigma_j \circ \varphi_j$, for $j=1,2$. Consider the \sh{} $\varphi = \varphi_1 \oplus \varphi_2 \colon \cA \to \cB_1 \oplus \cB_2$. Fix $0 < c < 1$ and let $\sigma$ be the tracial state on $\cB_1 \oplus \cB_2$ given by $\sigma(b_1,b_2) = c\sigma_1(b_1)+(1-c)\sigma_2(b_2)$, for $b_1 \in \cB_1$ and $b_2 \in \cB_2$. Then $c\tau_1 + (1-c) \tau_2 = \sigma \circ \varphi$, which belongs to $T_\mathrm{fin}(\cA)$.

Suppose, conversely, that $\tau_1, \tau_2$ belong to $T(\cA)$, and that $c \tau_1 + (1-c) \tau_2$ belongs to $T_\mathrm{fin}(\cA)$, for some $0 < c < 1$. Then $\cA/I_{c \tau_1 + (1-c) \tau_2}$ is finite dimensional. But $I_{c \tau_1 + (1-c) \tau_2} = I_{\tau_1} \cap I_{\tau_2}$, so $\cA/I_{\tau_1}$ and  $\cA/I_{\tau_2}$ are both finite dimensional, whence $\tau_1,\tau_2$ belong to $T_\mathrm{fin}(\cA)$.

The last claim follows from the fact that the closure of a face of any compact convex set is, again, a face.

(ii). It is well-known that $\tau$ is an extreme point in $T(\cA)$ if and only if $\pi_\tau(\cA)''$ is a factor. If $\pi_\tau(\cA)$ is finite dimensional, then this happens if and only if $\pi_\tau(\cA)$ is a full matrix algebra, whence $\tau$ is as desired. 

Let $\tau$ be an arbitrary finite dimensional trace on $\cA$, and write it as $\tau = \tau_0 \circ \varphi$, for some unital \sh{} $\varphi \colon \cA \to \cB$ onto some finite dimensional \Cs{} $\cB$, and  some tracial state $\tau_0$ on $\cB$. Write $\cB = \bigoplus_{j=1}^r \cB_j$, where each $\cB_j$ is a full matrix algebra equipped with tracial state $\mathrm{tr}_{B_j}$, and let $\pi_j \colon \cB \to \cB_j$ be the canonical projection. Then $\tau = \sum_{j=1}^r c_j \tau_j$, where $\tau_j = \mathrm{tr}_{\cB_j} \circ \pi_j \circ \varphi$ and $c_j = \tau(e_j)$, where $e_j \in \cB$ is the unit of $\cB_j$. 
\end{proof}

\noindent
We have the following inclusions of tracial states on any unital \Cs{} $\cA$:
$$T_\mathrm{fin}(\cA) \subseteq \overline{T_\mathrm{fin}(\cA)} \subseteq T_\mathrm{qd}(\cA) \subseteq T_\mathrm{am}(\cA) \subseteq T_\mathrm{hyp}(\cA) \subseteq T(\cA),$$
where $T_\mathrm{qd}(\cA)$, $T_\mathrm{am}(\cA)$ and $T_\mathrm{hyp}(\cA)$ are the sets of \emph{quasi-diagonal}, \emph{amenable}, respectively, \emph{hyperlinear} tracial states on $\cA$. Recall, e.g., from  \cite{Brown:MAMS}, see also the introduction of \cite{Schafhauser:qd}, that a tracial state $\tau$ on $\cA$ is \emph{hyperlinear} if it factors through an ultrapower $\cR^\omega$ of the hyperfinite II$_1$-factor $\cR$. Equivalently, $\tau$ is hyperlinear if $\pi_\tau(\cA)''$ embeds in a trace-preserving way into $\cR^\omega$.   If the embedding $\pi_\tau(\cA)'' \to \cR^\omega$ moreover can be chosen to admit a u.c.p.\ lift $\pi_\tau(\cA)'' \to \ell^\infty(\cR)$, then $\tau$ is \emph{amenable} (or \emph{liftable}, in the terminology of Kirchberg, \cite{Kir:T}). A tracial state $\tau$ on an \emph{exact} \Cs{} $\cA$ is amenable if and only if $\pi_\tau(\cA)''$ is hyperfinite, see  \cite[Corollary 4.3.6]{Brown:MAMS}. A trace $\tau$ is \emph{quasi-diagonal} if it factors through the \Cs{} $\prod_{n=1}^\infty M_{k_n}(\C)/\bigoplus_{k=1}^\infty M_{k_n}(\C)$ with a u.c.p.\ lift $\cA \to \prod_{n=1}^\infty M_{k_n}(\C)$,  for some sequence of integers $k_n \ge 1$. 

Each of the three sets $T_\mathrm{qd}(\cA)$, $T_\mathrm{am}(\cA)$, and $T_\mathrm{hyp}(\cA)$ is compact and convex. Kirchberg proved in \cite{Kir:T} that, moreover, $T_\mathrm{am}(\cA)$  is a face of $T(\cA)$.  The Connes Embedding Problem is equivalent to $T_\mathrm{hyp}(\cA)$ being equal to $T(\cA)$, for all \Cs s $\cA$. It is not known if  $T_\mathrm{qd}(\cA) =T_\mathrm{am}(\cA)$ in general, but in recent remarkable works, e.g.,  \cite{TWW:qd}, \cite{Gabe:qd} and \cite{Schafhauser:qd}, it  has been resolved  that amenable traces are quasi-diagonal in many important cases of interest. The inclusions $\overline{T_\mathrm{fin}(\cA)} \subseteq T_\mathrm{qd}(\cA)$ and $T_\mathrm{am}(\cA) \subseteq T_\mathrm{hyp}(\cA)$ are in general strict, even for residually finite dimensional \Cs s $\cA$, cf.\ Propositions~\ref{prop:Nate} and Example~\ref{ex:extend} (as well as the remarks above  Example~\ref{ex:extend}). In particular, $T_\mathrm{fin}(\cA)$ need not be weak$^*$-dense in $T(\cA)$, for residually finite dimensional \Cs s $\cA$.

Recall that a \Cs{} $\cA$ is \emph{residually finite dimensional} if it admits a separating family of finite dimensional representations $\varphi_i \colon \cA \to M_{k(i)}(\C)$, $i \in I$. The index set $I$ can be taken to be countable if $\cA$ is separable. Equivalently, $\cA$ is residually finite dimensional if and only if the set of finite dimensional traces on $\cA$ is \emph{separating}, in the sense that $\bigcap_{\tau \in T_\mathrm{fin}(\cA)} I_\tau = \{0\}$. We can sharpen this statement, as follows:

\begin{proposition} \label{prop:RFD-T_fin}
A unital \Cs{} $\cA$ is residually finite dimensional if and only if $\overline{T_\mathrm{fin}(\cA)}$ is separating. If $\cA$ is separable, then $\overline{T_\mathrm{fin}(\cA)}$ is separating if and only if it  contains a faithful trace.
\end{proposition}

\begin{proof} The first part of the statement follows from the remark above, and the fact that $\bigcap_{\tau \in \cT} I_\tau = \bigcap_{\tau \in \overline{\cT}} I_\tau$, for every subset $\cT$ of $T(\cA)$. 

Assume now that $\cA$ is separable and that  $\overline{T_\mathrm{fin}(\cA)}$ is separating. Observe first that for each positive element $a \in \cA$, there is $\tau \in \overline{T_\mathrm{fin}(\cA)}$ such that $\|a+I_\tau\| \ge \|a\|/2$. To see this we may assume that $\|a\|=1$. Let $g \colon [0,1] \to [0,1]$ be a continuous function which is zero on $[0,1/2]$ and with $g(1)=1$. Then $g(a)$ is positive and non-zero, so there is $\tau \in \overline{T_\mathrm{fin}(\cA)}$ such that $\tau(g(a)) > 0$. It follows that $g(a+I_\tau) = g(a)+I_\tau \ne 0$. This entails that $\|a+I_\tau\| > 1/2$.

Let $\{a_n\}_{n \ge 1}$ be a dense set of positive contractions in $\cA$. For each $n \ge 1$, find $\tau_n \in \overline{T_\mathrm{fin}(\cA)}$ such that $\|a_n + I_{\tau_n}\| \ge \|a_n\|/2$. Set $\tau = \sum_{n=1}^\infty 2^{-n} \tau_n$. Since $\overline{T_\mathrm{fin}(\cA)}$ is convex and weak$^*$-closed (and hence norm-closed), it follows that $\tau$ belongs to $\overline{T_\mathrm{fin}(\cA)}$. Moreover, $I_\tau = \bigcap_{n=1}^\infty I_{\tau_n}$, whence $\|a_n+I_\tau\| \ge \|a_n + I_{\tau_n}\| \ge \|a_n\|/2$, for all $n \ge 1$. This implies that $\|a+I_\tau\| \ge \|a\|/2$, for all positive contractions $a \in \cA$, from which we see that $\tau$ is faithful.
\end{proof}

\noindent
Kirchberg proved in \cite{Kir:T} that $\overline{T_{\mathrm{fin}}(C^*(G))} = T_{\mathrm{am}}(C^*(G))$, for all discrete  groups $G$ with Kazhdan's property (T); he also proved therein  that a  property (T) group is  residually finite if and only if it possesses the \emph{factorization property}. Hence, by Proposition~\ref{prop:RFD-T_fin}, a countably infinite property (T) group $G$ has residually finite dimensional group \Cs{} $C^*(G)$ if and only if $C^*(G)$ has a faithful amenable trace. As it turns out, there are no known examples of infinite property (T) groups where $C^*(G)$ is residually finite dimensional, or even quasi-diagonal, or where $C^*(G)$ has a faithful tracial state (amenable or not). Bekka proved in \cite{Bekka-Invent-07} that $C^*(G)$ has no faithful tracial state, and hence is not residually finite dimensional, when $G = \mathrm{SL}(n,\Z)$, for $n \ge 3$.

\begin{proposition}[N.\ Brown, {\cite[Corollary 4.3.8]{Brown:MAMS}}] \label{prop:Nate}
There exists an exact unital residually finite dimensional \Cs{} $\cA$ for which $T_\mathrm{am}(\cA) \ne T_{\mathrm{hyp}}(\cA)$. 
\end{proposition}

\noindent In more detail, Brown showed in \cite{Br-N:quasi} that there is a separable unital exact residually finite dimensional \Cs{} $\cA$ that surjects onto $C^*_\lambda(\mathbb{F}_2)$. The (unique) tracial state $\tau$ on $\cA$ that factors through $C^*_\lambda(\mathbb{F}_2)$ is not amenable, because $\pi_\tau(\cA)''$ (which is equal to the group von Neumann algebra $\mathcal{L}(\mathbb{F}_2)$) is not hyperfinite, and $\cA$ is exact. It is hyperfinite because $\mathcal{L}(\mathbb{F}_2)$ admits an embedding into $\cR^\omega$, as observed by Connes in \cite{Con:class}.

\begin{remark} \label{rem:not-closed}
We note that the  convex set of finite dimensional tracial states on a unital \Cs{}  is almost never closed. More precisely, if $\cA$ is a unital residually finite dimensional \Cs, then $T_{\mathrm{fin}}(\cA)$ is closed if and only if $\cA$ is finite dimensional. Indeed, if $\cA$ is infinite dimensional, then it admits a sequence $\{\pi_n\}_{n\ge 1}$ of pairwise inequivalent finite dimensional irreducible representations. For $n \ge 1$, set $\tau_n = \mathrm{tr}_{k(n)} \circ \pi_n$, where $k(n)$ is the dimension of the representation $\pi_n$. Then 
$\{\tau_n\}_{n \ge 1}$ is a sequence of distinct extreme points of $T_{\mathrm{fin}}(\cA)$, and $\tau := \sum_{n=1}^\infty 2^{-n} \tau_n$ belongs to the closure of $T_{\mathrm{fin}}(\cA)$, but not to $T_{\mathrm{fin}}(\cA)$ itself.
\end{remark}

\noindent
A separable  \Cs{} is residually finite dimensional if and only if it embeds into a \Cs{} of the form $\cM= \prod_{n=1}^\infty M_{k(n)}(\C)$, for some sequence $\{k(n)\}_{n \ge 1}$ of positive integers. The (non-separable) \Cs{} $\cM$ is itself residually finite dimensional. The following result  is contained in Ozawa, \cite[Theorem 8]{Ozawa:Dixmier}, but also follows from \cite{Wright:AW*}, as explained below:

\begin{proposition}  \label{prop:Ozawa}
The set $T_\mathrm{fin}(\cM)$ is weak$^*$-dense in $T(\cM)$, when $\cM = \prod_{n=1}^\infty M_{k(n)}(\C)$.
\end{proposition}

\begin{proof}  Since $\cM$ is an AW$^*$-algebra (in fact, a finite von Neumann algebra), we can use   \cite{Wright:AW*} to see that any tracial state on $\cM$ factors through the center $\cZ(\cM)$ of $\cM$, via the (unique) center-valued trace. The center  $\cZ(\cM)$ can be identified with  $C(\beta \N)$, the continuous functions on the Stone-\v{C}ech compactification of $\N$. Hence, tracial states on $\cM$ are in one-to-one correspondence with probability measures on $\beta \N$. Furthermore, finite dimensional traces correspond to convex combinations of Dirac measures at points of $\N$. Since $\N$ is dense in $\beta \N$, and since the convex hull of Dirac measures (at points of $\beta \N$) is dense in the set of all probability measures on $\beta \N$, we reach the desired conclusion. 
\end{proof}

\noindent Consider a unital embedding $\varphi$ of a unital residually finite dimensional \Cs{} $\cA$ into 
$\cM:= \prod_{n=1}^\infty M_{k(n)}(\C)$, for some sequence of positive integers $k(n) \ge 1$. Let $\pi_n$ denote the $n$th coordinate map $\cM \to M_{k(n)}(\C)$.  We say that the inclusion $\varphi$ is \emph{saturated} if 
\begin{equation} \label{eq:standard}
\Big(\bigoplus_{n=1}^N \pi_n\Big)(\cA) = \bigoplus_{n=1}^N M_{k(n)}(\C),
\end{equation} for all $N \ge 1$, and $\{\tau_n : n \ge 1\}$ is weak$^*$-dense in $\partial_e T(\cA) \cap T_\mathrm{fin}(\cA)$, where $\tau_n = \mathrm{tr}_{k(n)} \circ \pi_n \circ \varphi$. 

Each separable  residually finite dimensional  unital \Cs{} admits a saturated embedding into some $ \prod_{n=1}^\infty M_{k(n)}(\C)$. Indeed, if $\cA$ is separable, then $T(\cA)$ is separable, too, and hence so is 
$\partial_e T(\cA) \cap T_\mathrm{fin}(\cA)$. Pick a countable dense subset $\{\tau_n : n \ge 1 \}$ of this set. Then  $\tau_n = \mathrm{tr}_{k(n)} \circ \varphi_n$, for some surjective \sh{} $\varphi_n \colon \cA \to M_{k(n)}(\C)$, cf.\ Proposition~\ref{prop:fdtrace1}, and $\varphi := \bigoplus_{n \ge 1} \varphi_n$ is a saturated embedding of $\cA$. Injectivity of $\varphi$ follows from
$\mathrm{ker}(\varphi) = \bigcap_{n=1}^\infty I_{\tau_n} = \bigcap_{\tau \in T_{\mathrm{fin}}(\cA)} I_\tau = \{0\}$, where the second equality follows from density of $\{\tau_n : n \ge 1 \}$ in $\partial_e T(\cA) \cap T_\mathrm{fin}(\cA)$ and  Proposition~\ref{prop:fdtrace1}. 

\begin{proposition} \label{prop:extendtraces}
Let $\cA$ be a unital \Cs.
\begin{enumerate}
\item If $\tau \in T(\cA)$ factors through $\prod_{n=1}^\infty M_{k(n)}(\C)$, for some sequence of integers $k(n) \ge 1$, then $\tau \in \overline{T_{\mathrm{fin}}(\cA)}$.
\item If $\cA$ is a separable residually finite dimensional \Cs{} with saturated embedding $\varphi \colon \cA \to \prod_{n=1}^\infty M_{k(n)}(\C):= \cM$, then $\overline{T_{\mathrm{fin}}(\cA)}$ consists precisely of those traces $\tau$  on $\cA$ that extend to a trace on $\cM$ (in the sense that  $\tau = \tau' \circ \varphi$, for some $\tau' \in T(\cM)$).
\end{enumerate}
\end{proposition}

\begin{proof} Part (i) follows from Proposition~\ref{prop:Ozawa} and the fact that if $\tau'$ belongs to $T_{\mathrm{fin}}(\cM)$, then $\tau' \circ \varphi$ belongs to $T_{\mathrm{fin}}(\cA)$.

(ii). Denote, as above, the $n$th coordinate map $\cM \to M_{k(n)}(\C)$ by $\pi_n$. Each of the tracial states  $\mathrm{tr}_{k(n)} \circ \pi_n \circ \varphi$ extends to the tracial state $\mathrm{tr}_{k(n)}  \circ \pi_n$ on $\cM$. By assumption, $\{\mathrm{tr}_{k(n)}  \circ \pi_n \circ \varphi  : n \ge 1\}$ is dense in $\partial_e T(\cA) \cap T_\mathrm{fin}(\cA)$. The set of tracial states on $\cA$ that extend to a tracial state on $\cM$ is closed and convex. (Indeed, it is equal to the image of the 
continuous affine homomorphism $T(\cM) \to T(\cA)$ induced by the embedding $\varphi \colon \cA \to \cM$.) We may therefore conclude from Proposition~\ref{prop:fdtrace1} (ii) that each tracial state in the closure of $T_\mathrm{fin}(\cA)$ extends to a tracial state on $\cM$. The other inclusion follows from (i).
\end{proof}

\noindent Proposition~\ref{prop:extendtraces} raises the question when  a tracial state on a unital sub-\Cs{} $\cB$ of a unital \Cs{} $\cA$ can be extended to $\cA$. This is well-understood when $\cB \subseteq \cA$ are von Neumann algebras: each tracial state on $\cB$ extends to a tracial state on $\cA$ if and only if finite central projections in $\cA$ separate finite central projections in $\cB$, i.e., if $p,p'$ are distinct finite central projections in $\cB$, then there exists a finite central projection $q$ in $\cA$ such that $pq$ is zero and $p'q$ is non-zero, or vice versa. The corresponding question for \Cs s is more  subtle: Take $\cB$ to be a unital \Cs{} with a faithful extremal tracial state $\tau$. Then $\cB \subseteq \pi_\tau(\cB)''$, and $\pi_\tau(\cB)''$ is a type II$_1$-factor. Hence $\tau$ is the only tracial state on $\cB$ that extends to a tracial state on 
$\pi_\tau(\cB)''$, while $\cB$ need not have unique trace (even for simple \Cs s $\cB$). We pursue this issue  in Example~\ref{ex:extend} below. 

As it was remarked in \cite[Section 2.1]{ESS-C*-stablegroups} (see also relevant definitions therein), a matricially weakly semiprojective \Cs{}  is residually finite dimensional if and only if it is quasi-diagonal if and only if it is MF (defined in \cite{BlaKir:MF}). Every weakly semiprojective \Cs{} is also 
matricially weakly semiprojective. 

\begin{proposition} \label{prop:semiprojective}
Let $\cA$ be a unital (matricially) weakly semiprojective \Cs. Then 
$\overline{T_{\mathrm{fin}}(\cA)} = T_{\mathrm{qd}}(\cA)$. 
\end{proposition}

\begin{proof} If $\tau \in T_{\mathrm{qd}}(\cA)$, then $\tau$ factors through  $\cM:=\prod_{n=1}^\infty M_{k(n)}(\C) / \bigoplus_{n=1}^\infty M_{k(n)}(\C)$ via a \sh{} $\varphi$ and a tracial state $\tau_0$ on $\cM$, for some sequence of positive integers $k(n)$. 
Since $\cA$ is matricially weakly semiprojective, $\varphi$ lifts to a \sh{} $\psi$:
$$
\xymatrix@C+1pc@R+1pc{&  \prod_{n=1}^\infty M_{k(n)}(\C)\ar[d]^\pi & \\ \cA \ar@{-->}[ur]^-\psi \ar[r]_-\varphi & 
\prod_{n=1}^\infty M_{k(n)}(\C) / \bigoplus_{n=1}^\infty M_{k(n)}(\C) \ar[r]_-{\tau_0}& \C.}
$$
Hence $\tau$ factors through $\prod_{n=1}^\infty M_{k(n)}(\C)$, so it belongs to $\overline{T_{\mathrm{fin}}(\cA)}$ by Proposition~\ref{prop:extendtraces}~(i).
\end{proof}

\noindent There exists unital residually finite dimensional \Cs s $\cA$ for which $\overline{T_{\mathrm{fin}}(\cA)} \ne T_{\mathrm{qd}}(\cA)$, see \cite[Example 3.11]{HadShul:stability} by Hadwin and Shulman. As an  application of Proposition~\ref{prop:extendtraces}, we exhibit here a larger class of \Cs s with these properties.

\begin{example} \label{ex:extend}
Take a unital MF-algebra $\cB$ witnessed by an embedding $\varphi$ as in the diagram:
$$\xymatrix{ 0 \ar[r] & \bigoplus_{n \ge 1} M_{k(n)}(\C) \ar[r] \ar@{=}[d]& \cA \ar@{^{(}->}[d] \ar[r]^-{\rho \, = \, \varphi^{-1} \circ \pi} & \cB \ar[r] \ar[d]_\varphi  & 0\\
0 \ar[r] & \bigoplus_{n \ge 1} M_{k(n)}(\C) \ar[r]&  \prod_{n \ge 1} M_{k(n)}(\C)
  \ar[r]^-\pi &  \frac{ \prod_{n\ge 1} M_{k(n)}(\C)}{\bigoplus_{n\ge 1} M_{k(n)}(\C)} \ar[r]  & 0
.}$$
The pull-back \Cs{} $\cA := \pi^{-1}(\varphi(\cB))$ is then unital and residually finite dimensional.  If $\cB$ has no finite dimensional representations, then the inclusion of $\cA$ into $\prod_{n=1}^\infty M_{k(n)}(\C)$ is saturated. Each tracial state on $\cB$ gives rise to a tracial state on $\cA$ by composition with $\rho$. The resulting tracial state on $\cA$ will not always extend to a tracial state on  $\prod_{n=1}^\infty M_{k(n)}(\C)$, as shown below.

Take now $\cB$ to be a unital MF-algebra with no finite dimensional representations, and which admits a faithful quasi-diagonal tracial state $\tau$. Then there is a sequence $\{k(n)\}_{n \ge 1}$ of positive integers and u.c.p.\ maps $\mu_n \colon \cB \to M_{k(n)}(\C)$, $n \ge 1$, such that for all $a,b \in \cB$,
$$\lim_{n\to\infty} \|\mu_n(ab)-\mu_n(a)\mu_n(b)\| = 0, \quad \lim_{n\to\infty}\|\mu_n(a)\| = \|a\|, \quad \lim_{n\to\infty} \mathrm{tr}_{k(n)}(\mu_n(a)) = \tau(a).$$
Hence $\varphi := \pi \circ \bigoplus_{n \ge 1} \mu_n$ is an injective \sh{} as in the diagram above. Let also $\cA$ be as above. If $\tau'$ is a tracial state on $\cB$, then $\tau' \circ \rho$ extends to a trace on $ \prod_{n\ge 1} M_{k(n)}(\C)$ if and only if $\tau=\tau'$. Indeed, if $\sigma$ is a tracial state on $ \prod_{n\ge 1} M_{k(n)}(\C)$ that extends $\tau' \circ \rho$, then, e.g., by Proposition~\ref{prop:Ozawa} and its proof, $\sigma(\{x_n\}) = \omega(\{ \mathrm{tr}_{k(n)}(x_n)\})$, for all $\{x_n\}_{n \ge 1} \in \prod_{n\ge 1} M_{k(n)}(\C)$, for some state $\omega$ on $\ell^\infty(\N)$, which vanishes on $c_0(\N)$, since $\tau' \circ \rho$, and therefore also $\sigma$, vanish on  $\bigoplus_{n\ge 1} M_{k(n)}(\C)$. It follows that
$$\tau'(b) = (\tau' \circ \rho)(\{\mu_n(b)\}) =  \sigma(\{\mu_n(b)\}) = \omega( \{\mathrm{tr}_{k(n)}(\mu_n(b)\}) = \tau(b),$$
for all $b \in \cB$. 

We conclude that if $\tau' \in T(\cB)$ and $\tau' \ne \tau$, then $\tau' \circ \rho$ does not extend to $ \prod_{n\ge 1} M_{k(n)}(\C)$, whence $\tau' \circ \rho$ does not belong to $\overline{T_{\mathrm{fin}}(\cA)}$, cf.\ Proposition~\ref{prop:extendtraces}, while $\tau' \circ \rho$ does belong to $T_{\mathrm{qd}}(\cA)$ whenever $\cB$, moreover, is quasi-diagonal and nuclear, and $\tau'$ is faithful.
\end{example}

\noindent
We now turn our interest to the particular example of the unital universal free product $M_n(\C) *_\C M_n(\C)$ of two copies of the full matrix algebra $M_n(\C)$. It was shown in \cite{ExelLoring:free_products} that $M_n(\C) *_\C M_n(\C)$ is residually finite dimensional, while Blackadar proved that it is semiprojective, see \cite[Corollary 2.28 and Proposition 2.31]{Bla:shape}.

\begin{lemma} \label{lm:lift-M_n} Let $n \ge 1$ be an integer and let $\pi \colon \cA \to \cB$ be a surjective \sh{} between unital \Cs s $\cA$ and $\cB$. Suppose, furthermore, that the following conditions hold:
\begin{enumerate}
\item[\rm{(a)}] the unitary group of $\cB$ is connected;
\item[\rm{(b)}] whenever $p,q \in \cB$ are projections such that the $n$-fold direct sum of $p$ is equivalent to the $n$-fold direct sum of $q$, then $p \sim q$;
\item[\rm{(c)}] there is  a unital embedding of $M_n(\C)$ into $\cA$.
\end{enumerate}
Then any unital \sh{} $M_n(\C) \to \cB$ lifts to a unital \sh{} $M_n(\C) \to \cA$, and any unital \sh{} $M_n(\C) *_\C M_n(\C) \to \cB$ lifts to a unital \sh{} $M_n(\C) *_\C M_n(\C) \to \cA$.
\end{lemma}

\begin{proof} Fix a unital \sh{} $\beta \colon M_n(\C) \to \cB$, and pick any unital \sh{} $\alpha' \colon M_n(\C) \to \cA$, cf.\ (c). Set $\beta' = \pi \circ \alpha'$. It follows from assumption (b) that $\beta(e_{11}) \sim \beta'(e_{11})$, where $e_{ij}$, $1 \le i,j \le n$, are the standard matrix units for $M_n(\C)$. It is a well-known fact, see, e.g., 
\cite[Lemma 7.3.2(ii)]{RorLarLau:k-theory}, that $\beta$ and $\beta'$ are unitarily equivalent, i.e., there is a unitary $u \in \cB$ such that $u\beta'(a)u^* = \beta(a)$, for all $a \in \cA$. By (a), $u$ lifts to a unitary $v \in \cA$. It follows that $\alpha \colon \cA \to M_n(\C)$ given by $\alpha(a) = v\alpha'(a)v^*$, $a \in \cA$, is a lift of $\beta$.

By the universal property of free products, there is a bijective correspondence between unital \sh s from $M_n(\C) *_\C M_n(\C)$ into a given unital \Cs{} and pairs of unital \sh s from $M_n(\C)$ into the same unital \Cs. The second statement about $M_n(\C) *_\C M_n(\C)$  follows therefore from the first one.
\end{proof}

\begin{theorem} \label{thm:Tfin(MnMn)} The closure of $T_{\mathrm{fin}}(M_n(\C) *_\C M_n(\C))$ is equal to $T_{\mathrm{hyp}}(M_n(\C) *_\C M_n(\C))$.
\end{theorem}

\begin{proof} Let $\tau \in T_{\mathrm{hyp}}(M_n(\C) *_\C M_n(\C))$. By the definition of hyperlinear traces, there is a unital embedding $\varphi \colon M_n(\C) *_\C M_n(\C) \to \cR^\omega$ such that $\tau = \tau_{\cR^\omega} \circ \varphi$. Let $\cQ$ denote the universal UHF algebra, and view it as a dense subalgebra of the hyperfinite II$_1$-factor $\cR$, with respect to $\| \, \cdot \, \|_2$. Composing the inclusion $ \prod_{k=1}^\infty \cQ  \to  \prod_{k=1}^\infty \cR$ with the natural surjection from $ \prod_{k=1}^\infty \cR$ onto $\cR^\omega$ yields the \sh{} $\pi$ in the following diagram:
$$
\xymatrix{& \prod_{k=1}^\infty \cQ \ar[d]^-\pi \\ M_n(\C) *_\C M_n(\C) \ar@{-->}[ur]^-\psi \ar[r]_-\varphi & \cR^\omega.}
$$
Since $\cQ$ is $\| \, \cdot \, \|_2$-dense in $\cR$, we see that $\pi$ is surjective. The \sh{} $\varphi$ lifts to a \sh{} $\psi = \bigoplus_{k=1}^\infty \psi_k$, by Lemma~\ref{lm:lift-M_n}. Moreover,
$(\tau_{\cR^\omega} \circ \pi)(\{b_k\}_{k \ge 1}) = \lim_{k \to \omega} \tau_\cQ(b_k)$, for all $\{b_k\}_{k \ge 1} \in \prod_{k=1}^\infty \cQ$. It follows that
$$\tau = \tau_{\cR^\omega} \circ \varphi = \tau_{\cR^\omega} \circ \pi \circ \psi = \lim_{k \to \omega} \tau_\cQ \circ \psi_k.$$
This shows that $\tau$ is the limit of a net of tracial states that factor through the universal UHF-algebra $\cQ$. As every tracial state that factors through $\cQ$ is quasi-diagonal, and the set of quasi-diagonal traces is closed, we conclude that $\tau$ is quasi-diagonal. By Proposition~\ref{prop:semiprojective}, this completes the proof.
\end{proof}

\noindent The theorem above implies that the set of finite dimensional tracial states on $M_n(\C) *_\C M_n(\C)$ is weak$^*$-dense if the Connes Embedding Problem has an affirmative answer.

\section{Factorizable channels and the Connes Embedding Problem: a new viewpoint} \label{sec:Choi}

\noindent We give in this section a new characterization of factorizable channels in terms of certain properties of their Choi matrix, that bears resemblance with the matrices of quantum correlations that appear in Tsirelson's conjecture. From this viewpoint, we then establish a new link to  the Connes Embedding Problem. 

Keeping consistent notation with previous papers on this topic, let $T\colon M_n(\C) \to M_n(\C)$ be a linear map. One associates to it its Choi matrix
$$C_T = \sum_{i,j=1}^n e_{ij} \otimes T(e_{ij}) \in M_n(\C) \otimes M_n(\C),$$
where $e_{ij}$, $1 \le i,j \le n$, are, as before, the standard matrix units for $M_n(\C)$.
Choi's celebrated theorem, \cite{Choi:cp}, states that $T$ is completely positive if and only if $C_T$ is a positive matrix. Furthermore, we can recover $T$ from the matrix $C_T$ by the formula
\begin{equation} \label{eq:T-CT}
T(e_{ij}) = \sum_{k,\ell=1}^n C_T(i,j;k,\ell) \,  e_{k\ell}, \qquad 1 \le i,j  \le n,
\end{equation}
 where $C_T(i,j;k,\ell)$ are the matrix coefficients of $C_T$, cf.\ \eqref{eq:CT-coeff} below, which we briefly justify: Equip the vector space $M_n(\C)$ with the inner product $\langle \, \cdot \, , \, \cdot \, \rangle_{\Tr_n}$ coming from the standard  trace $\Tr_n$ on $M_n(\C)$. (We reserve the notation $\mathrm{tr}_n$  for the normalized trace on $M_n(\C)$.) The set of standard matrix units $\{e_{ij}\}$ is then an orthonormal basis for $M_n(\C)$, and 
\begin{equation} \label{eq:CT-coeff}
C_T(i,j;k,\ell) = \langle T(e_{ij}), e_{k\ell}\rangle_{\Tr_n} = \langle C_T, e_{ij} \otimes e_{k\ell} \rangle_{\Tr_n \otimes \Tr_n}.
\end{equation}

A unital completely positive trace-preserving map  $T \colon M_n(\C) \to M_n(\C)$ is \emph{factorizable}, cf.\ \cite{A-D:factorizable}, if there exist a finite von Neumann algebra $M$ with normal faithful tracial state $\tau_M$ and unital \sh s $\alpha,\beta \colon M_n(\C) \to M$ such that $T = \beta^* \circ \alpha$, where $\beta^* \colon M \to M_n(\C)$ is the adjoint of $\beta$. The map $\beta^*$ is formally defined by the identity
$\langle \beta(x), y \rangle_{\tau_M} = \langle x, \beta^*(y) \rangle_{\mathrm{tr_n}}$,  for $x \in M_n(\C)$ and
$y \in M$,
and it is obtained  by composing the (unique) trace-preserving conditional expectation $E \colon M \to \beta(M_n(\C))$ with $\beta^{-1}$, see \cite{HaaMusat:CMP-2011}. In this case, $T$ is said to \emph{exactly factor through} $(M,\tau_M)$. 

As explained in \cite{HaaMusat:CMP-2011}, if $T$ factors through $(M,\tau_M)$, then we may write $M = M_n(\C) \otimes N$, for some ancillary finite von Neumann algebra $(N,\tau_N)$, and we may take $\beta$ to be given by $\beta(x) = x \otimes 1_N$, for $x \in M_n(\C)$. In this case, $\alpha(x) = u( x \otimes 1_N)u^*$, for some unitary $u \in M_n(\C) \otimes N$, and $T(x) = ( \mathrm{id}_{n} \otimes \tau_N)(u(x \otimes 1_N)u^*)$, for 
$x \in M_n(\C)$. This gives a more transparent  definition of $T$ being factorizable. The finite von Neumann algebra $N$ above is called the \emph{ancilla}. The reader should be warned that the ancilla is far from being unique, and determining the ``minimal'' ancilla for a given factorizable map $T$ seems to be a difficult task.

We shall now rephrase the notion of factorizability of a linear map $T \colon M_n(\C) \to M_n(\C)$ in terms of a certain property of its associated Choi matrix. 

\begin{proposition} \label{prop:coordinates} 
Let $n \ge 2$ be an integer, and let $T \colon M_n(\C) \to M_n(\C)$ be a linear map with Choi matrix $C_T =\big[C_T(i,j;k,\ell)\big]_{(i,k),(j,\ell)}$ as above. Then the following are equivalent:
\begin{enumerate}
\item $T$ is factorizable.
\item There is a von Neumann algebra $M$ with normal faithful tracial state $\tau_M$, a unital \sh{} $\alpha \colon M_n(\C) \to M$,  and a set of matrix units $\{f_{ij}\}_{i,j=1}^n$  in $M$, so that 
$$T(x) = \sum_{i,j=1}^n n \,  \langle \alpha(x), f_{ij} \rangle_{\tau_M} \; e_{ij}, \qquad x \in M_n(\C).$$
\item There is a von Neumann algebra $M$ with normal faithful tracial state $\tau_M$ and sets of  matrix units $\{f_{ij}\}_{i,j=1}^n$ and $\{g_{ij}\}_{i,j=1}^n$ in $M$, so that $$n^{-1} C_T(i,j;k,\ell) =\tau_M(f_{k \ell}^* \, g_{ij}), \qquad 1 \le i,j,k,\ell \le n.$$
\item There is a tracial state $\tau$ on the unital free product \Cs{} $M_n(\C) \! *_\C \! M_n(\C)$, so that 
\begin{equation} \label{eq:T-tau}
n^{-1} C_T(i,j;k,\ell) =\tau(\iota_2(e_{k\ell})^*\iota_1(e_{ij})), \qquad 1 \le i,j,k,\ell \le n,
\end{equation} 
where $\iota_1$ and $\iota_2$ are the two canonical inclusions of $M_n(\C)$ into 
$M_n(\C) \! *_\C \! M_n(\C)$. 
\end{enumerate}
\end{proposition}

\begin{proof} (i) $\Leftrightarrow$ (ii). Suppose that $T$ is factors through a finite von Neumann algebra $M$ equipped with  normal faithful tracial state $\tau_M$ via unital \sh s $\alpha, \beta \colon M_n(\C) \to M$. Let $E \colon M \to \beta(M_n(\C))$ be the trace-preserving conditional expectation. Set $f_{ij} = \beta(e_{ij})$, for $1 \le i,j \le n$. Then $\{\sqrt{n} \, f_{ij}\}$ is an orthonormal basis for $\beta(M_n(\C))$ with respect to the inner product $\langle \, \cdot \, , \, \cdot \, \rangle_{\tau_M}$ on $M$, induced by  $\tau_M$. It follows that 
$$E(x) = \sum_{i,j=1}^n n \langle x, f_{ij} \rangle_{\tau_M} \, f_{ij}, \qquad x \in M.$$
This proves (ii), since $T = \beta^{-1} \circ E \circ \alpha$.

For the converse direction, let $\beta \colon M_n(\C) \to M$ be given by $\beta(e_{ij}) = f_{ij}$, $1 \le i,j \le n$. By reversing the argument above, one can verify that $T = \beta^{-1} \circ E \circ \alpha = \beta^* \circ \alpha$.

(ii) $\Leftrightarrow$ (iii). Let $M$, $\tau_M$ and $\alpha$  be as in (ii). For $1 \le i,j \le n$,  set $g_{ij} = \alpha(e_{ij})$. Then
\begin{eqnarray*}
C_T(i,j;k,\ell) &=&  \langle T(e_{ij}), e_{k\ell}\rangle_{\Tr_n}  = n \sum_{s,t=1}^n \langle \alpha(e_{ij}), f_{st} \rangle_{\tau_M} \cdot  \langle e_{st}, e_{k\ell}\rangle_{\Tr_n}
\\ &=& n \, \langle g_{ij}, f_{k\ell} \rangle_{\tau_M} = n  \, \tau_M(f_{k\ell}^* \, g_{ij}), \qquad 1 \le i,j,k,\ell \le n.
\end{eqnarray*}
Conversely, if (iii) holds, then define $\alpha \colon M_n(\C) \to M$ by $\alpha(e_{ij}) = g_{ij}$, $1 \le i,j \le n$. Then
$$T(e_{ij}) = \sum_{k,\ell=1}^n C_T(i,j;k,\ell) \, e_{k,\ell} =n   \sum_{k,\ell=1}^n \langle g_{ij}, f_{k\ell} \rangle_{\tau_M}  \, 
e_{k\ell} =n \,  \sum_{k,\ell=1}^n \langle \alpha(e_{ij}), f_{k\ell} \rangle_{\tau_M} \,  e_{k\ell}.
$$

(iii) $\Leftrightarrow$ (iv). Assuming that (iii) holds, let $\pi \colon M_n(\C) \! *_\C \! M_n(\C) \to M$ be the  \sh{} satisfying $\pi(\iota_1(e_{ij})) = g_{ij}$ and $\pi(\iota_2(e_{ij})) = f_{ij}$, $1 \le i,j \le n$. Set $\tau = \tau_M \circ \pi$. Then $\tau$ is a tracial state on $M_n(\C) \! *_\C \! M_n(\C)$, satisfying $\tau(\iota_2(e_{k\ell})^*\iota_1(e_{ij})) = \tau_M(f_{k \ell}^* \, g_{ij})$, $1 \le i,j,k,\ell \le n$.

Conversely, if (iv) holds with respect to some tracial state $\tau$ on $M_n(\C) \! *_\C \! M_n(\C)$, let $M$ be the finite von Neumann algebra $\pi_\tau(M_n(\C) \! *_\C \! M_n(\C))''$, equipped with the extension $\tau_M$ of $\tau$ to $M$. Then (iii) holds with $g_{ij} = \pi_\tau(\iota_1(e_{ij}))$ and $f_{ij} = \pi_\tau(\iota_2(e_{ij}))$, for  
$1 \le i,j \le n$.
\end{proof}

\noindent Note that by (iii) one can identify the set $\cFM(n)$ of factorizable maps on $M_n(\C)$, via their Choi matrix, with the set consisting of complex correlation matrices $ \big[\tau(f_{k\ell}^* \, g_{ij})\big]_{i,j,k,\ell}$, where $\{f_{k\ell}\}_{k,\ell}$ and $\{g_{ij}\}_{i,j}$ are systems $n \times n$ matrix units in some von Neumann algebra $(M,\tau)$. This latter set bears resemblance to the set of matrices of quantum correlations arising, for example, in Tsirelson's conjecture.

Using Proposition~\ref{prop:coordinates}, for $n \ge 2$ we can define a map 
$\Phi \colon T(M_n(\C) \! *_\C \! M_n(\C)) \to \cFM(n)$
by letting $T=\Phi(\tau)$ be the factorizable channel determined, via its Choi matrix, by \eqref{eq:T-tau}, for each tracial state $\tau$ on $M_n(\C) \! *_\C \! M_n(\C)$. More precisely, for all $x \in M_n(\C)$,
\begin{equation} \label{eq:Phi}
\Phi(\tau)(x) = \sum_{i,j=1}^n  n \, \tau(\iota_2(e_{ij})^* \, \iota_1(x)) \, e_{ij}.
\end{equation}

Following the notation of \cite{MusRor:infdim}, denote by $\cFM_{\mathrm{fin}}(n)$ the set of maps in $\cFM(n)$ that admit a factorization through a finite dimensional \Cs.

\begin{proposition} \label{prop:Phi} The map $\Phi \colon T(M_n(\C) \! *_\C \! M_n(\C)) \to \cFM(n)$ defined above is continuous, affine and surjective. Moreover,
\begin{enumerate}
\item $\mathcal{FM}_\mathrm{fin}(n) = \Phi\big(  T_\mathrm{fin}(M_n(\C) \! *_\C \! M_n(\C)) \big)$;
\item $\overline{\mathcal{FM}_\mathrm{fin}(n)} = \Phi\big( \overline{ T_\mathrm{fin}(M_n(\C) \! *_\C \! M_n(\C))} \big) = \Phi\big(  T_\mathrm{hyp}(M_n(\C) \! *_\C \! M_n(\C)) \big)$.
\end{enumerate}
 \end{proposition}
 
 \begin{proof} Surjectivity of $\Phi$ follows from Proposition~\ref{prop:coordinates}. To prove it is continuous and affine, it suffices to show that the map $\tau \mapsto \Phi(\tau)(x)$ is continuous and affine, for all $x \in M_n(\C)$. This follows easily from \eqref{eq:Phi}.
 
 (i). If $\tau$ belongs to $T_\mathrm{fin}(M_n(\C) \! *_\C \! M_n(\C))$, then $M:= \pi_\tau(M_n(\C) \! *_\C \! M_n(\C))''$ is finite di\-men\-sional. It follows from the proofs of the implications (iv) $\Rightarrow$ (iii) $\Rightarrow$ (ii) $\Rightarrow$ (i) in Proposition~\ref{prop:coordinates} that $T = \Phi(\tau)$ admits a factorization through $(M, \tau)$, so $T$ belongs to $\mathcal{FM}_\mathrm{fin}(n)$.
 
 Likewise, if $T$ belongs to $\mathcal{FM}_\mathrm{fin}(n)$, then we can take the finite von Neumann algebra $M$ with normal faithful tracial state $\tau_M$ in (iii) of Proposition~\ref{prop:coordinates} to be finite dimensional. Let $\pi \colon  M_n(\C) \! *_\C \! M_n(\C) \to M$ be as in the proof of (iii) $\Rightarrow$ (iv) in Proposition~\ref{prop:coordinates}, and let $\tau = \tau_M \circ \pi$. Then $\tau$ is a tracial state on $M_n(\C) \! *_\C \! M_n(\C) $ with kernel $I_{\tau'} = \mathrm{ker}(\pi)$. Hence $(M_n(\C) \! *_\C \! M_n(\C))/I_{\tau'}$ is finite dimensional, so $\tau \in  T_\mathrm{fin}(M_n(\C) \! *_\C \! M_n(\C))$. It follows from the proof of  (iii) $\Rightarrow$ (iv)  in Proposition~\ref{prop:coordinates} that $T = \Phi(\tau)$. 
 
Finally, (ii) follows from (i), continuity of $\Phi$, compactness of $\overline{ T_\mathrm{fin}(M_n(\C) \! *_\C \! M_n(\C))}$, and Theorem~\ref{thm:Tfin(MnMn)}.
 \end{proof}

\begin{remark} Proposition~\ref{prop:Phi} provides a direct proof, avoiding ultraproduct arguments, of the well-known fact that the set  $\cFM(n)$ is a compact convex subset of the normed vector space of all linear maps on $M_n(\C)$. 
\end{remark}

\noindent
Note that the map $\Phi$ is not injective. More precisely, if we let $V_n$ denote the $n^4$-dimensional operator subspace of $M_n(\C) \! *_\C \! M_n(\C)$ spanned by $\{\iota_1(x)\iota_2(y) : x,y \in M_n(\C)\}$, then, for $\tau,\tau' \in T(M_n(\C) \! *_\C \! M_n(\C))$,
\begin{equation}\label{eq:En}
\Phi(\tau) = \Phi(\tau') \; \;  \text{if and only if} \; \;  \tau|_{V_n} = \tau'|_{V_n}. \qquad 
\end{equation}

The next corollary extends and sheds new light on \cite[Theorem 3.7]{HaaMusat:CMP-2015}, which states that (i) and (ii) below are equivalent.

\begin{corollary} \label{cor:CEP} The following statements are equivalent:
\begin{enumerate}
\item The Connes Embedding Problem has an affirmative answer;
\item $\mathcal{FM}_\mathrm{fin}(n)$ is dense in $\cFM(n)$, for all $n \ge 3$;
\item $T_\mathrm{hyp}(M_n(\C) \! *_\C \! M_n(\C))  =  T(M_n(\C) \! *_\C \! M_n(\C))$, for all $n \ge 2$;
\item For each $n \ge 2$, and each $\tau$ in $T(M_n(\C) \! *_\C \! M_n(\C))$, there is $\tau'$ in $T_\mathrm{hyp}(M_n(\C) \! *_\C \! M_n(\C))$ such that $\tau|_{V_n} = \tau'|_{V_n}$.
\end{enumerate}
\end{corollary}

\begin{proof} It is clear that an affirmative answer to the Connes Embedding Problem is equivalent to all traces on all \Cs s being hyperlinear, thus proving (i) $\Rightarrow$ (iii), while the implication (iii) $\Rightarrow$ (iv) is trivial. It follows from Proposition~\ref{prop:Phi} (ii) and \eqref{eq:En} that (iv) $\Rightarrow$ (ii). Finally, (ii) $\Rightarrow$ (i) is contained in \cite[Theorem 3.7]{HaaMusat:CMP-2015}.
\end{proof}

\begin{remark} Suppose that $\tau$ is a tracial state on $M_n(\C) \! *_\C \! M_n(\C)$, and that $T= \Phi(\tau)$ is the corresponding factorizable map on $M_n(\C)$. Then, by the proof of Proposition~\ref{prop:coordinates}, we see that $T$ admits a factorization through the finite von Neumann algebra $M = \pi_\tau(M_n(\C) \! *_\C \! M_n(\C))''$, equipped with the trace $\tau$. In particular, we see that $M$ admits an embedding into $\cR^\omega$ if and only if $\tau$ is hyperlinear. It was shown in \cite{HaaMusat:CMP-2015} that $T$ belongs to $\mathcal{FM}_\mathrm{fin}(n)$ if and only if it admits a factorization through a finite von Neumann algebra that embeds into $\cR^\omega$.
\end{remark}

\begin{remark} \label{rem:JP}
J.\ Peterson mentioned to us that one can prove the implication (iii) $\Rightarrow$ (i) of Corollary~\ref{cor:CEP} directly as follows: Assume that (iii) holds and that $(M,\tau)$ is a separable II$_1$-factor. Upon replacing $M$ by $M \otimes \cR$, we may assume that $M$ is singly generated, and hence generated by two self-adjoint elements $a$ and $b$, that can be taken to be contractions. Take sequences $\{a_n\}_{n \ge 1}$ and $\{b_n\}_{n \ge 1}$ of self-adjoint contractions converging with respect to $\| \, \cdot \, \|_2$ to $a$ and $b$, respectively, so that $C^*(1_M,a_n)$ and $C^*(1_M,b_n)$ admit unital embeddings (necessarily trace-preserving) into $M_n(\C)$. (Such unital embeddings exist precisely when $a_n$ and $b_n$ are of the form $\sum_{j=1}^n\lambda_j e_j$, for some real numbers $\lambda_j$ and some pairwise orthogonal and pairwise equivalent projections $e_1, \dots, e_n$ summing up to $1$.) Then $C^*(1_M,a_n,b_n)$ admits a unital embedding into $M_n(\C) \! *_\C \! M_n(\C)$, that is trace-preserving with respect to some tracial state $\tau$ on $M_n(\C) \! *_\C \! M_n(\C)$, which, by assumption, is hyperlinear. This shows that $C^*(1_M,a_n,b_n)$ admits a unital trace-preserving embedding into $\cR^\omega$.  Consequently, $M$ embeds into the double ultrapower $(\cR^\omega)^\omega$, and therefore into $\cR^\omega$, by a diagonal argument.
\end{remark}

\noindent We end this paper with a result concerning the structure of the simplex 
$T(M_n(\C) *_\C M_n(\C))$, and a related result describing which unital \Cs s are quotients of $M_n(\C) \! *_\C \! M_n(\C)$, or, equivalently, which unital \Cs s can be generated by two copies of $M_n(\C)$.  Recall that a unital \Cs{} is generated by $n \ge 1$ elements if and only if it is generated by $2n$ self-adjoint elements; and if a unital \Cs{} is generated by $n$ self-adjoint elements, then it is also generated by $n$ unitary elements.

\begin{proposition} \label{prop:fin_gen}
Let $\cA$ be a unital \Cs, and let $n \ge 2$ be an integer. 
\begin{enumerate}
\item If there exists a unital surjective \sh{} $M_n(\C) \! *_\C \! M_n(\C) \to M_n(\C) \otimes \cA$, then $\cA$ is generated by at most $n^2$ elements. 
\item If $\cA$ is generated by $n-1$ unitaries, then there exists a unital surjective \sh{} $M_n(\C) \! *_\C \! M_n(\C) \to M_n(\C) \otimes \cA$.
\end{enumerate}
\end{proposition}

\begin{proof}
(i). The unital \sh{} $\varphi \colon M_n(\C) \! *_\C \! M_n(\C) \to M_n(\C) \otimes \cA$ is determined by two unital \sh s $\alpha,\beta \colon M_n(\C) \to M_n(\C) \otimes \cA$, and we may take $\alpha(x) = x \otimes 1_\cA$, for $x \in M_n(C)$. Now, $M_n(\C)$ is singly generated, say by an element $g \in M_n(\C)$, and
$$\beta(g) = \sum_{i,j=1}^n e_{ij} \otimes g_{ij},$$
for some elements $g_{ij} \in \cA$, $1 \le i,j \le n$. Since $\varphi$ is surjective, it follows that $\cA$ must be generated by the set $\{g_{ij} : 1 \le i,j \le n\}$. 

(ii). Suppose that $\cA$ is generated by unitaries $u_2, \dots, u_n$ in $\cA$. Set $f_{11} = e_{11} \otimes 1_\cA$ and  $f_{1j} = e_{1j} \otimes u_j$, for $2 \le j \le n$. Note that $f_{1j}f_{1j}^* = e_{11} \otimes 1_\cA$, and that $f_{1j}^*f_{1j} =e_{jj} \otimes 1_\cA$, for $1 \le j \le n$. Further, set $f_{ij} = f_{1i}^*f_{1j}$, $1 \le i,j \le n$. Observe that $\{f_{ij}\}$ is a set of $n\times n$ matrix units in $M_n(\C) \otimes \cA$. Hence there exists a unital \sh{} $\beta \colon M_n(\C) \to  M_n(\C) \otimes \cA$ satisfying $\beta(e_{ij}) = f_{ij}$, $1 \le i,j \le n$. Let $\gamma \colon M_n(\C) \! *_\C \! M_n(\C) \to M_n(\C) \otimes \cA$ be  determined by $\alpha$ and $\beta$ (i.e., $\gamma(\iota_1(x)) = \alpha(x)$ and $\gamma(\iota_2(x)) = \beta(x)$, for $x \in M_n(\C)$). It is then easy to see that  $1 \otimes u_j$ belongs to the image of $\gamma$, for $2 \le j \le n$, and that $M_n(\C) \otimes 1_\cA$ is contained in the image of $\gamma$. This shows that $\gamma$ is surjective.
\end{proof}

\noindent
It follows in particular that $M_n(\C) \otimes \cA$ is a quotient of $M_n(\C) \! *_\C \! M_n(\C)$, for every singly generated unital \Cs{} $\cA$, when $n \ge 3$. It was shown in \cite{ThielWin:generator} that every unital separable $\cZ$-stable \Cs{} is singly generated.  It is easy to see that every finite dimensional \Cs{} is singly generated, so $M_n(\C) \otimes \cA$ is generated by two copies of $M_n(\C)$, whenever $\cA$ is finite dimensional and $n \ge 3$. 

\begin{remark} We know, e.g., from  Remark~\ref{rem:not-closed} that $T_{\mathrm{fin}}(M_n(\C) \! *_\C \! M_n(\C))$ is not closed, for all $n \ge 2$. For $n \ge 11$, this also follows from Proposition~\ref{prop:Phi} and the main result from \cite{MusRor:infdim}, which states that $\cFM_{\mathrm{fin}}(n)$ is non-closed.

One can exhibit many traces in $T(M_n(\C) \! *_\C \! M_n(\C))$, and also in $\overline{T_{\mathrm{fin}}(M_n(\C) \! *_\C \! M_n(\C))}$, which are of type II$_1$. Indeed, take any unital separable tracial \Cs{} $(\cA, \tau)$. Then, by Proposition~\ref{prop:fin_gen}, there is a trace $\tau'$ on $M_n(\C) \! *_\C \! M_n(\C)$ that factors through the trace  $\tau \otimes {\mathrm{tr}_n} \otimes \tau_\cZ$ on $\cA \otimes M_n(\C) \otimes \cZ$. The trace $\tau'$ is always of type II$_1$, it is a factor trace if  $\tau$ is, and $\tau'$ belongs to the closure of $T_{\mathrm{fin}}(M_n(\C) \! *_\C \! M_n(\C))$, which equals $T_{\mathrm{hyp}}(M_n(\C) \! *_\C \! M_n(\C))$, if $\pi_\tau(\cA)''$ embeds into $\cR^\omega$.
\end{remark}

\noindent As an interesting application of Proposition~\ref{prop:fin_gen}, we show next that the trace simplex of $M_n(\C) \! *_\C \! M_n(\C)$, for $n \ge 3$, is as large as possible. 
For the proof of this result we make use of the following elementary fact: Any surjective unital \sh{} $\varphi \colon \cA \to \cB$ between unital \Cs s $\cA$ and $\cB$ induces an affine continuous injective map $T(\varphi) \colon T(\cB) \to T(\cA)$, by $T(\tau) = \tau \circ \varphi$, for $\tau \in T(\cB)$. Moreover, $T(\varphi)$ maps extreme points of $T(\cB)$ into extreme points of $T(\cA)$, and hence faces of $T(\cB)$ onto faces of $T(\cA)$. Indeed, if $\tau$ is an extreme point of $T(\cB)$, then $\pi_\tau(\cB)''$ is a factor. As $\cB = \varphi(\cA)$, we infer that $\pi_{\tau \circ \varphi}(\cA) = \pi_\tau(\cB)$, so $\pi_{\tau \circ \varphi}(\cA)'' = \pi_\tau(\cB)''$ is a factor, which implies that $T(\tau) = \tau \circ \varphi$ is an extreme point. 

\begin{theorem} \label{prop:Poulsen}
Let $n \ge 3$ be an integer. Then each metrizable Choquet simplex is affinely homeomorphic to a (closed) face of $T(M_n(\C) \! *_\C \! M_n(\C))$.
\end{theorem}

\begin{proof} Let $S$ be a metrizable Choquet simplex. Then there is a simple infinite dimensional unital AF-algebra $\cA$ such that $T(\cA)$ is affinely homeomorphic to $S$, see, e.g., \cite{Eff:AF} or \cite{LazLin-1971}. Every simple infinite dimensional unital AF-algebra is $\cZ$-absorbing, see \cite[Theorem 5]{JiangSu:Z}, and hence singly generated. It follows from Proposition~\ref{prop:fin_gen} above that there is a unital surjective \sh{} $\varphi \colon M_n(\C) \! *_\C \! M_n(\C) \to M_n(\C) \otimes \cA$, which, in turn, induces an injective affine continuous map $T(\varphi) \colon T(\cA) \to T(M_n(\C) \! *_\C \! M_n(\C))$. As remarked above, the image of $T(\varphi)$ is a face of $T(M_n(\C) \! *_\C \! M_n(\C))$ which is affinely homeomorphic to $S$.
 \end{proof}

\noindent It was shown in \cite[Theorems 2.3, 2.5 and 2.11]{LinOlsStern:Poulsen} that a metrizable Choquet simplex $S$ is  the Poulsen simplex if and only if the following two conditions hold: 
\begin{enumerate}
\item Each metrizable Choquet simplex is affinely homeomorphic to a face of $S$.
\item (Homogeneity) For every pair $F,F'$ of faces of $S$ with $\mathrm{dim}(F) = \mathrm{dim}(F') < \infty$, there is an affine homemorphism of $S$ that maps $F$ onto $F'$. 
\end{enumerate}
We would like to point out that property (i) by itself  does not characterize the Poulsen simplex, and hence one cannot conclude from 
Proposition~\ref{prop:Poulsen} that $T(M_n(\C) \! *_\C \! M_n(\C))$  is the Poulsen simplex. Indeed, if $S$ is the Poulsen simplex embedded in a locally convex topological  vector space $V$, then the suspension $S' := \{ (ts, 1-t) : s \in S, \, 0 \le t \le 1\} \subseteq V \times \R$ of $S$ is a Choquet simplex that contains $S$ as a face, but it is not itself the Poulsen simplex, as the extreme points are not dense.

\vspace{.3cm} \noindent \emph{Acknowledgements:}
We thank James Gabe for fruitful discussions about traces on residually finite dimensional \Cs s, Erik Alfsen for sharing with us his insight on the Poulsen simplex, Jesse Peterson for pointing out the argument in Remark~\ref{rem:JP}, and Taka Ozawa for suggesting the reference \cite{Wright:AW*}. We have also benefitted from discussions with Marius Dadarlat.

{\small{
\bibliographystyle{amsplain}
%\bibliography{operator}

\providecommand{\bysame}{\leavevmode\hbox to3em{\hrulefill}\thinspace}
\providecommand{\MR}{\relax\ifhmode\unskip\space\fi MR }
\providecommand{\MRhref}[2]{%
  \href{http://www.ams.org/mathscinet-getitem?mr=#1}{#2}
}
\providecommand{\href}[2]{#2}

}}

\vspace{1cm}

\noindent
\begin{tabular}{ll}
Magdalena Musat & Mikael R\o rdam \\
Department of Mathematical Sciences & Department of Mathematical Sciences\\
University of Copenhagen & University of Copenhagen\\ 
Universitetsparken 5, DK-2100, Copenhagen \O & Universitetsparken 5, DK-2100, Copenhagen \O \\
Denmark & Denmark\\
musat@math.ku.dk & rordam@math.ku.dk
\end{tabular}

\end{document}